\documentclass[11pt,reqno]{amsart}
\usepackage{microtype}
\usepackage[margin=3cm]{geometry}

\usepackage[dvipsnames]{xcolor}
\usepackage
[colorlinks=true,linkcolor=Maroon,citecolor=OliveGreen,backref]
{hyperref}
\usepackage{amsmath, amsfonts,amsthm,amssymb,verbatim,
    color,multirow,booktabs,mathdots,bm}
\usepackage[shortlabels]{enumitem}
\setlist[enumerate]{label={(\arabic*)}}
\usepackage[capitalize]{cleveref}
\crefname{equation}{}{}

\usepackage[abbrev, alphabetic]{amsrefs}  

\usepackage{listings}
\lstdefinestyle{mystyle}{
    backgroundcolor=\color{white},
    commentstyle=\color{gray},
    keywordstyle=\color{blue},
    numberstyle=\tiny\color{gray},
    stringstyle=\color{orange},
    basicstyle=\ttfamily\footnotesize,
    breaklines=true,
    captionpos=b,
    numbers=left,
    numbersep=5pt,
    showstringspaces=false,
    tabsize=4
}

\numberwithin{equation}{section}
\newtheorem{theorem}{Theorem}[section]
\newtheorem{lemma}[theorem]{Lemma}

\newtheorem{proposition}[theorem]{Proposition}

\newtheorem*{theorem*}{Theorem}
\newtheorem*{question}{Question}

\theoremstyle{definition}
\newtheorem{definition}[theorem]{Definition}

\theoremstyle{remark}
\newtheorem{remark}[theorem]{Remark}

\newcommand{\Alt}{\mathop{\mathrm{Alt}}}

\newcommand{\F}{\mathbb{F}}
\newcommand{\E}{\mathbb{E}}

\newcommand\br[1]{{\left(#1\right)}}
\newcommand\sqbr[1]{{\left[#1\right]}}
\newcommand\gbinom[2]{{#1 \brack #2}}
\newcommand\gen[1]{\left\langle#1\right\rangle}

\def\cH{\mathcal{H}} 

 \definecolor{mycolor}{rgb}{0.55,0.0,0.16}
  \definecolor{myred}{rgb}{0.75,0.0,0.16}
  \definecolor{mygreen}{rgb}{0.0,0.4,0.16}
  \definecolor{myviolet}{rgb}{1,0,1}
   \definecolor{mypink}{rgb}{0.67,0,0.47}

    \usepackage{todonotes}

\numberwithin{equation}{section}

\makeatletter
\@namedef{subjclassname@2020}{%
  \textup{2020} Mathematics Subject Classification}
\makeatother


\author{Sean Eberhard}
\address{\parbox{\linewidth}{Sean Eberhard, Mathematics Institute, Zeeman Building, University of Warwick,
Coventry CV4~7AL, United Kingdom \vspace{0.1cm}}}
\email{sean.eberhard@warwick.ac.uk}

\author{Luca Sabatini}
\address{\parbox{\linewidth}{Luca Sabatini, Mathematics Institute, Zeeman Building, University of Warwick,
Coventry CV4~7AL, United Kingdom \vspace{0.1cm}}}
\email{luca.sabatini@warwick.ac.uk}

\thanks{The authors are supported by the Royal Society.}

\begin{document}
\title[Probabilistic construction of some extremal $p$-groups]
{Probabilistic construction\\of some extremal $p$-groups}

\begin{abstract}
    A $p$-group $G$ is called \emph{ab-maximal} if $|H : H'| < |G:G'|$ for every proper subgroup $H$ of $G$.
    Similarly, $G$ is called \emph{$d$-maximal} if $d(H) < d(G)$ for every proper subgroup $H$ of $G$, where $d(H)$ is the minimal number of generators of $H$.
    If $G$ is ab-maximal and nonabelian then $|G:G'| \ge p^3 |G'|$, while if $G$ is $d$-maximal and $p \ne 2$ then $|G:G'| \ge p^2 |G'|$.
    Answering questions of Gonz\'alez-S\'anchez--Klopsch and Lisi--Sabatini, for all $p$ we construct infinitely many ab-maximal $p$-groups of class $2$ with $|G:G'| = p^3 |G'|$,
    and infinitely many $d$-maximal $p$-groups of class $2$ with $|G:G'| = p^2 |G'|$.
    The construction is probabilistic and based on the degeneracy of random alternating bilinear maps on subspaces.
    It is notable however that in the ab-maximal case we do not have a high-probability result but rather in a suitable sense the proportion of class-$2$ groups with $|G:G'| = p^n$ and $|G'| = p^{n-3}$ that are ab-maximal is close to $1/e$ (where $e$ is the base of the natural logarithm).
\end{abstract}

\maketitle

\section{Introduction}

If $G$ is a finite group, we write $d(G)$ for the minimal size of a generating set.
We call $G$ \emph{$d$-maximal} if $d(H) < d(G)$ for every proper subgroup $H$ of $G$.
This terminology was introduced by Kahn in \cite{Kah91}.
Although $d$-maximal groups are not nilpotent in general \cite{LSS24}, in this paper we only consider $d$-maximal $p$-groups.
In this case $d(G) = \log_p |G : G' G^p|$, so the condition above can equivalently be written as $|H : H' H^p| < |G : G' G^p|$ for all $H < G$.
Laffey~\cite{Laf73} proved that if $G$ is a $d$-maximal $p$-group with $p$ odd then $G$ has class at most two, and $G/G'$ and $G'$ are both elementary abelian.

A closely related concept is ab-maximality.
We say that a finite group $G$ is \emph{ab-maximal} if $|H:H'|<|G:G'|$ for every proper subgroup $H$ of $G$.
It was observed in \cite{Sab22} that ab-maximal groups are nilpotent of class at most two,
and that every group of order $n$ contains an ab-maximal subgroup of order at least $n^{c/\log\log n}$ for some constant $c > 0$.
This result is sharp even if we replace ``ab-maximal'' with ``solvable'',
and motivates the study of ab-maximal groups in greater detail.

For a group of prime exponent, $d$-maximality and ab-maximality are equivalent.
Trivial examples include elementary abelian groups and extraspecial groups,
and it is natural to ask how large the derived subgroup can be in these groups in general.
In the literature, this has been asked a couple of times \cite{GK11,LS24}.
For example, Question C in \cite{LS24} asks whether
every ab-maximal group $G$ satisfies $|G:G'| \geq |G|^\delta$ for some fixed $\delta>0.5$.
We answer this question in the negative.

\begin{theorem} \label{thExTop}
    For each prime $p$ and $n \geq 160$,
    there exists an ab-maximal $p$-group $G$ (of exponent $p$ if $p \ne 2$)
    such that $|G : G'|=p^n$ and $|G'|=p^{n-3}$.
\end{theorem}

In general if $G$ is a nonabelian ab-maximal $p$-group then $|G : G'| \ge p^3 |G'|$, so the multiplicative gap $p^3$ is the smallest possible.
The inequality $n \ge 160$ is just a convenience, and needed mainly if $p = 2$.
If $p \ge 13$ it suffices to take $n \ge 12$.
In truth we expect $n \ge 6$ is sufficient for all $p$.

To prove the above theorem we use standard ideas to reduce to a question in linear algebra, which we solve using probabilistic methods.
This is not the first time that wild $p$-groups of class two have been constructed using a random alternating map.
For example, using essentially this idea Ol'shanskii~\cite{Ols78} showed the existence of groups of order $p^n$ where all abelian subgroups have order at most $p^m$ where $m \sim \sqrt{8n}$.
See also \cite{BGH87} or \cite[Theorem 1.9]{HPPS24} for applications of the same method.
Notably, however, \Cref{thExTop} is \emph{not} obtained from a high-probability result, but rather we show roughly that a random candidate $p$-group (namely, a $p$-group $G$ of Frattini class two and $|G : G'| = p^3 |G'|$) is ab-maximal with probability close to $1 / e$.
In linear algebraic terms, the basic reason for this curious result is that a random alternating bilinear map $B \colon \F_p^n \times \F_p^n \to \F_p^{n-3}$ has a number of totally isotropic $3$-dimensional subspaces that is asymptotically distributed as a Poisson distribution.

Using the same method we construct $d$-maximal groups with large derived subgroup.
This case is a bit easier and in fact we do obtain a high-probability result.
In particular we give a complete positive answer to Question~4.5 in \cite{GK11}.

\begin{theorem} \label{thExDMax}
    For each prime $p$ and $n \geq 2$,
    there exists a $d$-maximal $p$-group $G$
    such that $|G : G'|=p^n$ and $|G'| = p^{n-2}$.
\end{theorem}

\section{Extremal ab-maximal and $d$-maximal groups}

\subsection{ab-maximal groups}

Recall that a finite group $G$ is called ab-maximal if $|H:H'|<|G:G'|$ for all proper subgroups $H$ of $G$.
(Here $G'=[G,G]$ denotes the derived subgroup of $G$.)
These groups are called $\gamma_2$-maximal in \cite{LS24}, with reference to the more general notion of $w$-maximality introduced in \cite{GK11}.
In this subsection $p$ is a prime and $G$ is a nonabelian ab-maximal $p$-group.
We do not require $p$ to be odd.

The following result of Thompson \cite{Tho69} is crucial.
See \cite{LS24}*{Lemma~2.9} for a wide generalization.

\begin{theorem}[Thompson] \label{thTho}
    Let $G$ be an ab-maximal $p$-group.
    Then
    $G' \leqslant Z(G)$.
\end{theorem}

It is trivial that $|G : G'| > |Z(G)|$, and a slightly more involved argument provides the following sharp inequality.

\begin{lemma} \label{lemExtTop}
    Let $G$ be a nonabelian ab-maximal $p$-group.
    Then
    $|G : G'| \geq |Z(G)| p^3$.
\end{lemma}
\begin{proof}
    Let $A < G$ be a maximal abelian subgroup.
    Then $C_G(A) = A$ and $Z(G) < A$.
    We claim that $A / Z(G)$ is noncyclic.
    Otherwise, $A = Z(G) \langle a \rangle$ for some $a \in A$, and $A = C_G(A) = C_G(a)$.
    Moreover $|a^G| \le |G'|$ (since $a^G \subseteq aG'$), so $|A| = |C_G(a)| = |G| / |a^G| \ge |G : G'|$, in contradiction with the ab-maximality of $G$.
    Therefore $A / Z(G)$ is noncyclic and $|A : Z(G)| \ge p^2$.
    Thus by ab-maximality we have $|G : G'| > |A| \ge |Z(G)| p^2$,
    which implies $|G:G'| \ge |Z(G)| p^3$.
\end{proof}

\begin{remark}
    \Cref{lemExtTop} conflicts with some statements in \cite{GK11}*{Section~4}.
    In fact there is a mistake in \cite{GK11}*{Example~4.4}.
    With reference to the notation in that example, the subalgebra $\gen{e_1, e_3, z_1, z_2} \le L$ is abelian, so $G = \exp(L)$ is not $d$-maximal as claimed there.
    Incidentally, one can check that $k(G) = 105$ if $p=3$, $k(G) = p^4 + p^2 - 1$ if $p \equiv 2 \pmod 3$, and $k(G) = p^4 + 2p^3 - p^2 - 2p + 1$ if $p \equiv 1 \pmod 3$, so $G$ is one of 6.4.2--6.4.4 in the list of Vaughan--Lee~\cite{vaughanlee2015}.
\end{remark}

\begin{definition}
    An \emph{extremal ab-maximal $p$-group} is a nonabelian ab-maximal $p$-group $G$ such that $|G : G'| = |G'| p^3$.
\end{definition}

Note that if $G$ is an extremal ab-maximal $p$-group then $|G| = |G'|^2 p^3$.

\begin{lemma}
    If $G$ is an extremal ab-maximal $p$-group, then
    \begin{enumerate}
        \item $G' = Z(G)$;
        \item $G/G'$ and $G'$ are elementary abelian;
        \item $G$ has exponent $p$ or $p^2$.
    \end{enumerate}
\end{lemma}
\begin{proof}
    Part (1) follows from \Cref{thTho} and \Cref{lemExtTop}.
    For (2), suppose that $xG'$ is an element of $G/G'$ of order $p^2$.
    The map $yG' \mapsto [x,y]$ defines a homomorphism $G/G' \to G'$.
    Since $|G/G'| = |G'| p^3$, there are at least $p^3$ cosets $yG' \in G/G'$ such that $[x,y] = 1$,
    and in particular there is some such $yG' \notin \langle xG' \rangle$.
    It follows that $\gen{x,y}G'$ is an abelian subgroup of $P$ of order at least $|G'| p^3$,
    in contradiction with the ab-maximality of $G$.
    Thus $G/G'$ is elementary abelian, and so is $G'$ since it is generated by commutators.
    Finally (3) follows from (2).
\end{proof}

Examples of extremal ab-maximal $p$-groups include the extraspecial groups of order $p^5$. Curiously, there are no examples of order $p^7$.

\begin{lemma}
    \label{lem:n=5-excluded}
    There is no extremal ab-maximal group of order $p^7$.
\end{lemma}
\begin{proof}
    Suppose $G$ were such a group, so that $G' = Z(G) \cong \F_p^2$ and $G/G' \cong \F_p^5$.
    Let $V =\F_p^5$ and $W =\F_p^2$.
    The commutator map induces a nondegenerate alternating bilinear map $V\times V \to W$.
    By an observation of Higman and Alperin~\cite{alperin}*{Theorem~3}, any such map has a $3$-dimensional totally isotropic subspace.
    It follows that $G$ has an abelian subgroup of order $p^5$, contrary to ab-maximality.
\end{proof}

\subsection{$d$-maximal groups}

In this subsection assume $p$ is an odd prime and $G$ is a nonabelian $d$-maximal $p$-group.
Adapting Thompson's proof of \Cref{thTho}, Laffey~\cite{Laf73} proved the following:

\begin{theorem}[Laffey] \label{thLaffey}
    Let $G$ be a $d$-maximal $p$-group, where $p$ is odd.
    Then
    $G' = \Phi(G) \leqslant Z(G)$.
\end{theorem}

It follows from \Cref{thLaffey} that $G/G'$ and $G'$ are elementary abelian, so the exponent of $G$ is either $p$ or $p^2$.
The next result is an analogue of \Cref{lemExtTop} for $d$-maximal groups.

\begin{lemma}[\cite{GK11}*{Proposition~4.1}] \label{lemExtDMax}
    Let $G$ be a nonabelian $d$-maximal $p$-group, where $p$ is odd.
    Then
$|G:G'| \geq |G'| p^2$.
\end{lemma}

\begin{remark}
    \Cref{thLaffey,lemExtDMax} fail for $p=2$.
    Some $d$-maximal groups of order $2^8$ and class $3$ are known \cite{Min96},
    while the quaternion group of order $8$ satisfies $|G:G'|=4$ and $|G'|=2$.
\end{remark}

Although the inequality in \Cref{lemExtDMax} is weaker than that in \Cref{lemExtTop}, it is also sharp.
A nontrivial example is the $d$-maximal group
\[
    G = \gen{x, y, z: x^p = y^p = z^{p^2} = [x, z] = [y, z] = 1, [x, y] = z^p}.
\]
Here $|G|=p^4$, $Z(G) = \gen z$, and $G' = \gen {z^p}$.

\begin{definition}
    An \emph{extremal $d$-maximal $p$-group} is a nonabelian $d$-maximal $p$-group $G$ such that $|G : G'| = |G'| p^2$.
\end{definition}

From \Cref{lemExtTop} it is clear that an extremal $d$-maximal group must have exponent $p^2$.

\section{Reduction to linear algebra}

Let $\cH_p(n, m)$ be the class of $p$-groups $G$ such that $\Phi(G)$ is central and elementary abelian with index $p^n$ and order $p^m$.
These groups were famously used by Higman~\cite{Hig60} to give a lower bound for the number of groups of a given prime-power order.
Higman showed that $\cH_p(n, m) = \emptyset$ for $m > n(n+1)/2$, that $\cH_p(n, n(n+1)/2)$ contains a single group $H$, and that every group in $\cH_p(n, m)$ is a quotient of $H$ by a subgroup of $\Phi(H)$.
One may define $H$ to be the quotient of the free group $F = F_n$ by the fully invariant subgroup generated by the elements
\[
    x^{p^2}, \quad [x,y]^p, \quad [x, y, z] \qquad (x, y, z \in F).
\]
If $H$ has generators $h_1, \dots, h_n$ then $\Phi(H)$ has generators
\[
    h_i^p \quad (1 \le i \le n), \qquad [h_i, h_j] \quad (1 \le i < j \le n).
\]

Let $V$ and $W$ be vector spaces over some field, and let $B \colon V \times V \to W$ be a bilinear map.
If $U \leqslant V$ is a subspace we write $B(U, U)$ for the span of the image of $B$ on $U \times U$.
We call $U$ \emph{totally isotropic} (for $B$) if $B(U, U) = 0$,
and we call $B$ \emph{alternating} if every $1$-dimensional subspace is (totally) isotropic.
The space of alternating bilinear maps $B \colon V \times V \to W$ is denoted by $\Alt(V, W)$.

If $p = 2$ then a map $F \colon V \to W$ is called a \emph{quadratic map} associated to $B$ if $F(x + y) = F(x) + F(y) + B(x, y)$ identically, and we write $F(U)$ for the span of the image of $F$ on $U \leqslant V$.
In this case we have $B(U, U) \leqslant F(U)$ for all $U \leqslant V$.
We say that $U$ is \emph{singular} (for $F$) if $F(U) = 0$.
Singularity is stronger than isotropy in general.

If $G \in \cH_p(n, m)$ then the commutator map descends to an alternating bilinear map $B \colon V \times V \to W$ where $V = G / \Phi(G) \cong \F_p^n$ and $W = \Phi(G) \cong \F_p^m$.
Further, the $p$-power map $x \mapsto x^p$ descends to a map $F \colon V \to W$ which is linear if $p \ne 2$ and quadratic and associated to $B$ if $p = 2$.
Since $W$ is by definition $\Phi(G) = G' G^p$, together these two maps satisfy $B(V, V) + F(V) = W$.

Higman's argument shows that the passage of the previous paragraph is reversible.
Suppose $V = \F_p^n$, $W = \F_p^m$, $B \colon V \times V \to W$ is bilinear and alternating, and $F \colon V \to W$ is linear if $p \ne 2$ and quadratic and associated to $B$ if $p = 2$.
Assume that $B(V, V) + F(V) = W$.
Write $B = (B_1, \dots, B_m)$ and $F = (F_1, \dots, F_m)$ with respect to a basis of $W$.
Then there is a group $G = G_{B,F} \in \cH_p(n, m)$ defined by generators $x_1, \dots, x_n, y_1, \dots, y_m$ and relations
\begin{align*}
    [x_i, x_j] = \prod_{k=1}^m y_k^{B_k(x_i, x_j)} , \qquad
    x_i^p = \prod_{k=1}^m y_k^{F_k(x_i)}, \qquad
    [x_i, y_j] = 1, \qquad y_i^p = 1.
\end{align*}
This correspondence is almost a special case of the well-known Baer correspondence~\cite{Bae38}, except that we allow $p = 2$
(cf.~\cite{rossmann-voll}*{Section~2.4}).

\begin{proposition}
    \label{prop:reduction}
    Let $V = \F_p^n$ and $W = \F_p^m$.
    Let $B \colon V \times V \to W$ be an alternating bilinear map and let $F \colon V \to W$ be linear if $p \ne 2$ and quadratic and associated to $B$ if $p = 2$.
    The following are equivalent.
    \begin{enumerate}
        \item $B(V, V) = W$ and $G_{B, F} \in \cH_p(n, m)$ is ab-maximal;
        \item every proper subspace $U < V$ satisfies
        \[
            \dim W / B(U, U) \> < \> \dim V / U.
        \]
    \end{enumerate}
    Similarly, the following are equivalent.
    \begin{enumerate}[resume]
        \item $B(V, V) = W$ and $G_{B, F} \in \cH_p(n, m)$ is $d$-maximal;
        \item every proper subspace $U < V$ satisfies
        \[
        \dim W / \br{B(U, U) + F(U)} \> < \> \dim V / U.
        \]
    \end{enumerate}
\end{proposition}
\begin{proof}
    Suppose (1) holds, and let $G = G_{B, F}$.
    Since $G'$ corresponds to $B(V, V)$ and $\Phi(G)$ corresponds to $W$, we have $G' = \Phi(G)$.
    Since $G$ is ab-maximal, $|H:H'| < |G:G'|$ for every $H < G$.
    If $U < V$, then $U$ corresponds to a subgroup $H < G$ containing $\Phi(G)$ and $B(U, U)$ corresponds to $H'$, so the ab-maximality condition implies (2).

    Now suppose (2) holds. If $U < V$ is a hyperplane, then
    \[
        \dim W / B(U, U) < \dim V / U = 1,
    \]
    so $B(V, V) \geqslant B(U, U) = W$. We also have from (2) that $|H:H'| < |G:G'|$ for every proper subgroup $H < G$ containing $\Phi(G) = G'$, and in general, if $H < G$ then $HG' < G$ and
    \[
        |H:H'| = |H:(HG')'| \le |HG' : (HG')'| < |G:G'|,
    \]
    so $G$ is ab-maximal. Thus (1) holds.

    The proof of $(3) \iff (4)$ is almost identical, using
    \[
        p^{d(H)} = |H:\Phi(H)| = |H:H'H^p|
    \]
    and the correspondence between subgroups $H$ containing $\Phi(G)$ and subspaces $U \leqslant V$, with $\Phi(H)$ corresponding to $B(U, U) + F(U)$.
    Let us only explain why (4) implies $B(V, V) = W$.
    If $B(V, V) \ne W$ then $B(V, V) \le W_0$ for some hyperplane $W_0 < W$. Let $U = F^{-1}(W_0)$, and note that $U$ is a linear subspace of $V$ (even if $p = 2$, since $B(V, V) \le W_0$).
    Moreover, since $W_0$ is a hyperplane, $\dim V/U \le 1$.
    But $B(U, U) + F(U) \le B(V, V) + W_0 = W_0 < V$, in contradiction with (4).
\end{proof}

\section{Totally isotropic sections of random alternating maps}

In this section, we move into the linear algebra at the heart of \Cref{thExTop} and \Cref{thExDMax}.
In view of \Cref{prop:reduction}, the following proposition completes the proof of \Cref{thExTop} (if $p = 2$ let $F$ be any quadratic map associated to $B$, and if $p \ne 2$ let $F = 0$).
We remark that \Cref{prop:alt-1}(2) improves \cite{BGH87}*{Main~Theorem} in the special case $k = n-3$.%
\footnote{
In the notation in \cite{BGH87}, we show $d(\F_p, n, n-3) = 2$ for large $n$.
More generally, it seems likely that a similar proof shows that their bound $d(F, n, k) \le (2n+k) / (k+2)$ is always strict when $F$ is finite.
Incidentally, the proof of \cite{BGH87}*{Main Theorem} is incomplete in multiple ways.
At the bottom of p.~271 the authors implicitly rely on the irreducibility of affine space, which is true only if the ground field is infinite, while the assertion at the top of p.~274 is incorrect if $k$ is odd. See~\cite{MO-question} for discussion.
Moreover the authors of \cite{BGH87} seem to have been unaware of the related work of Ol'shanskii~\cite{Ols78}.
}

\begin{proposition}   \label{prop:alt-1}
    Let $p$ be a prime and let $n \ge 5$.
    Let $V = \F_p^n$ and $W = \F_p^{n-3}$.
    If $B \in \Alt(V, W)$ is given, call a proper subspace $H < V$ \emph{bad (for $B$)} if
    \[\dim W / B(H, H) \> \ge \> \dim V / H .\]
    Let $B \in \Alt(V, W)$ be chosen uniformly at random.
    Then
    \begin{enumerate}
        \item the expected number of bad subspaces of dimension at least $4$ tends to $0$ as either $n, p \to \infty$;
        \item as $n \to\infty$, the number of bad (i.e., totally isotropic) $3$-dimensional subspaces is asymptotically Poisson with parameter
        $$
            \lambda_p \> = \> (1 - p^{-1})^{-1} (1 - p^{-2})^{-1} (1 - p^{-3})^{-1} ;
        $$
        \item the probability that there is no bad subspace is positive provided that either $\min(p, n) \ge 12$ or $n \ge 160$,
        and tends to $\exp(-\lambda_p)$ if $p$ is fixed and $n \to \infty$.
    \end{enumerate}
\end{proposition}

The fact that a $3$-dimensional subspace $H < V$ is bad if and only if it is totally isotropic is clear from the definition, since $\dim W / B(H, H) = n - 3 - \dim B(H, H)$ and $\dim V / H = n - 3$.

Throughout this section we write $V$ for $\F_p^n$,
$V \wedge V$ for the alternating square of $V$, and $\wedge$ for the universal alternating map $V \times V \to V \wedge V$.
Recall that $\dim V \wedge V = \binom{n}{2}$.
Alternating bilinear maps $V \times V \to W$ are in natural bijection with linear maps $V \wedge V \to W$.
If $H \leqslant V$ then $H \wedge H$ is naturally a subspace of $V \wedge V$.

We need the following technical lemma.

\begin{lemma} \label{lem:H_i^H_i}
    Let $H_1, \dots, H_k \leqslant V$ be distinct $3$-dimensional subspaces. Then
    \[
        \dim \sum_{i=1}^k H_i \wedge H_i \> \ge \> \dim \sum_{i=1}^k H_i,
    \]
    with equality if and only if $H_1, \dots, H_k$ are linearly independent, i.e., $\dim \sum_{i=1}^k H_i = 3k$.
\end{lemma}
\begin{proof}
    We use induction on $k$, the base case $k = 0$ being trivial.
    Let $k \ge 1$ and $U = \sum_{i=1}^k H_i$.
    Also, let $W = \sum_{i=1}^{k-1} H_i$ and $r = \dim (U / W) \in \{0, 1, 2, 3\}$.
    Let $x, y, z$ be a basis for $H_k$, of which the first $r$ elements are linearly independent modulo $W$,
    and the remaining $3-r$ elements are in $H_k \cap W$.
    We distinguish four cases according to $r$:

    \emph{Case $r = 3$}: In this case $x \wedge y$, $x \wedge z$ and $y \wedge z$ are linearly independent modulo $W \wedge W$,
    so by induction
    \[
        \dim \sum_{i=1}^k H_i \wedge H_i = \dim \sum_{i=1}^{k-1} H_i \wedge H_i + 3 \ge \dim W + 3 = \dim U .
    \]
    Equality holds if and only if $\dim \sum_{i=1}^{k-1} H_i \wedge H_i = \dim W$,
    and by induction this happens if and only if  $H_1, \dots, H_{k-1}$ are linearly independent.

    \emph{Case $r = 2$:} Even in this case, $x \wedge y$, $x \wedge z$ and $y \wedge z$ are linearly independent modulo $W \wedge W$.
    Indeed, choosing a basis $w_1, \dots, w_d$ for $W$ with $w_1 = z$, a basis for $U \wedge U$ modulo $W \wedge W$ is given by the elements of the form $x \wedge y$, $x \wedge w_i$, $y \wedge w_i$, including $x \wedge y, x \wedge z, y \wedge z$.
    It follows by induction that
    \[
        \dim \sum_{i=1}^k H_i \wedge H_i = \dim \sum_{i=1}^{k-1} H_i \wedge H_i + 3 \ge \dim W + 3 = \dim U + 1 .
    \]
    The equality with $\dim U$ is clearly impossible in this case.

    \emph{Case $r = 1$:} An argument as in the previous case shows that $x \wedge y$ and $x \wedge z$ are linearly independent modulo $W \wedge W$, so by induction
    \[
        \dim \sum_{i=1}^k H_i \wedge H_i \ge \dim \sum_{i=1}^{k-1} H_i \wedge H_i + 2 \ge \dim W + 2 = \dim U + 1.
    \]

    \emph{Case $r = 0$:} By induction,
    \[
        \dim \sum_{i=1}^k H_i \wedge H_i \ge \dim \sum_{i=1}^{k-1} H_i \wedge H_i \ge \dim W = \dim U.
    \]
    Equality can hold only if $H_1, \dots, H_{k-1}$ are linearly independent
    and $H_k \wedge H_k \leqslant \sum_{i=1}^{k-1} H_i \wedge H_i$.
    For $x \in H_k$, write $x = \sum_{i=1}^{k-1} x_i$ where $x_i \in H_i$.
    Then for $x, y \in H_k$ we have
    \[
        x \wedge y = \sum_{i,j=1}^{k-1} x_i \wedge y_j
        = \sum_{i=1}^{k-1} x_i \wedge y_i + \sum_{1 \le i < j \le k-1} (x_i \wedge y_j - y_i \wedge x_j) \in \sum_{i=1}^{k-1} H_i \wedge H_i.
    \]
    Since $H_1, \dots , H_{k-1}$ are linearly independent, the subspaces $H_i \wedge H_j \leqslant V \wedge V$ are linearly independent for $1 \le i \le j \le k-1$.
    It therefore follows from the above equation that $x_i \wedge y_j = y_i \wedge x_j$ for all $i < j$
    and
    \[
        x \wedge y = \sum_{i=1}^{k-1} x_i \wedge y_i.
    \]
    Let $x, y \in H_k$ be linearly independent.
    Then $x \wedge y \ne 0$, so $x_i \wedge y_i \ne 0$ for some $i$, so $x_i$ and $y_i$ are nonzero and not parallel.
    Since $x_i \wedge y_j = y_i \wedge x_j$ we must have $x_j = y_j = 0$ for all $j \ne i$.
    Since this holds for all $y \notin \gen{x}$, it follows that $H_k = H_i$, contrary to the hypothesis that $H_1, \dots, H_k$ are distinct.
\end{proof}

Below we use the standard notation
\[
    \gbinom{n}{d}_p \> = \> \prod_{i=0}^{d-1} \frac{p^{n-i}-1}{p^{d-i} -1}
\]
(the ``Gaussian binomial coefficient'') for the number of $d$-dimensional subspaces of $\F_p^n$
(we will drop the $p$ subscript when it should be clear).
We recall the well-known bounds
\begin{equation}
    \label{eq:gbinom-bounds}
    p^{d(n-d)} \le \gbinom{n}{d}_p \le (1 - 1/p - 1/p^2)^{-1} p^{d(n-d)} < 4 p^{d(n-d)}
\end{equation}
(see for example \cite{neumann-praeger}*{Lemma~3.5})
Note also that
 \begin{align*}
 \gbinom{n}{3}_p & \> = \>  \frac{(p^n - 1)(p^{n-1} - 1)(p^{n-2}-1)}{(p^3-1)(p^2-1)(p-1)} \\ & \> = \>
\lambda_p \, p^{3(n-3)} \br{1 + O(1/p^{n-2})}. \>
\end{align*}
For $k \geq 1$, we also use the falling factorial notation
\[
    (n)_k \> = \>  n (n-1) \cdots (n-k+1) .
\]
For example, $\br{\gbinom{n}{3}}_k$ is the number of ordered $k$-tuples of distinct $3$-dimensional subspaces.
Finally, we note that $\Alt(V, W)$ is a vector space over $\F_p$ of dimension $\binom n 2 \dim W$, where $n = \dim V$, and we note that if $K \le W$ then $\Alt(V, K)$ is naturally a subspace of $\Alt(V, W)$, while if $H \le V$ then $\Alt(H, K)$ is naturally a quotient of $\Alt(V, W)$.

\begin{proof}[Proof of \Cref{prop:alt-1}]
    Note that a subspace $H < V$ is bad for $B$ if and only if $\dim B(H, H) \le \dim H - 3$.
    In particular there are no bad subspaces of dimension less than $3$, and a $3$-dimensional subspace is bad if and only if it is totally isotropic.
    For $3 \le h < n$, let $N_h$ be the number of $h$-dimensional bad subspaces, and let $N_{\geqslant4} = \sum_{h=4}^{n-1} N_h$.

    First we consider $N_{\geqslant4}$.
    Fix $H < V$ with $\dim H = h \geq 4$.
    For any $K \leqslant W$ with $\dim K = h - 3$, we have
    \[
        \Pr \br{B(H, H) \le K} = \frac{|\Alt(H, K)|}{|\Alt(H, W)|} = p^{-\binom{h}{2}(n-h)}.
    \]
    By summing over all choices of $K$, we obtain by a union bound
    \[
        \Pr \br{\dim B(H, H) \le h - 3}
        \le \sum_{\dim K = h-3} \Pr\br{B(H, H) \le K}
        = \gbinom{n-3}{h-3} p^{-\binom{h}{2}(n-h)},
    \]
    and then by summing over $H$ it follows that
    \begin{align*}
        \E\sqbr{N_h}
        \le \gbinom{n}{h} \gbinom{n-3}{h-3} p^{-\binom{h}{2}(n-h)}
        \le 16 p^{h(n-h) + (h-3)(n-h) - \binom{h}{2}(n-h)}
        = 16 p^{f(h)},
    \end{align*}
    where $f(h) = (h-2)(h-3)(h-n)/2$.
    Since $f$ is a cubic function with roots at $2$, $3$, and $n$, it is clear that it has no local maximum greater than $3$, so its maximum in the range $[4, n-1]$ is attained at one of the endpoints (recall that $n \ge 5$).
    Moreover $f(4) = 4-n$ and $f(n-1) = -(n-3)(n-4)/2 \le 4-n$ since $n \ge 5$.
    It follows that $f(h) \le 4-n$ on $[4, n-1]$.
    Therefore, by Markov's inequality,
    \[
        \Pr\br{N_{\geqslant4} > 0} \le \E\sqbr{N_{\geqslant4}} = \sum_{h=4}^{n-1} \E\sqbr{N_h} \le 16 (n-4) p^{4-n} ,
    \]
    from which (1) follows easily.

    For (2), we will use the method of moments (see \cite{Bil95}*{Sec.~30}) to show that $N_3$ has a Poisson distribution.
    It suffices to prove that $\E\sqbr{(N_3)_k} \to \lambda_p^k$ as $n \to \infty$ for each $k \ge 1$.
    Note that $(N_3)_k$ is the number of ordered $k$-tuples of distinct totally isotropic $3$-dimensional subspaces $H_1, \ldots, H_k$ in $V$.
    Equivalently, bijectively associating $B$ with a linear map $\tilde B : V \wedge V \to W$,
    $\tilde B$ should be zero on the whole subspace $\sum_{i=1}^k H_i \wedge H_i$ of $V \wedge V$.
    For fixed linear subspaces $H_1, \dots, H_k \le V$, it is therefore clear that
    \begin{align*}
        \Pr\br{H_1, \dots, H_k~\text{are totally isotropic}}
        &= \Pr\br{\tilde B~\text{is zero on}~\sum_{i=1}^k H_i \wedge H_i} \\
        &= p^{-(n-3) \dim \sum_{i=1}^k H_i \wedge H_i},
    \end{align*}
    and it follows that
    \[
        \E\sqbr{(N_3)_k}
        = \sum_{\substack{H_1, \dots, H_k \\ \text{distinct $3$-dim subspaces}}} p^{-(n-3) \dim \sum_{i=1}^k H_i \wedge H_i}.
    \]
    Since $\dim \sum_{i=1}^k H_i \wedge H_i \le 3k$ for all $H_1, \dots, H_k$ we have
    \[
        \E\sqbr{(N_3)_k} \ge \br{\gbinom{n}{3}}_k \>p^{-(n-3)3k},
    \]
    and
    \[
        \E\sqbr{(N_3)_k} - \br{\gbinom{n}{3}}_k p^{-(n-3)3k}
        \le \sum_{\substack{H_1, \dots, H_k\\ \text{distinct $3$-dim subspaces} \\ \dim \sum_{i=1}^k H_i < 3k}}
        p^{-(n-3) \dim \sum_{i=1}^k H_i \wedge H_i} .
    \]
    The sum is empty if $k = 1$, so we may assume $k \ge 2$.
    We split the sum according to value of $d = \dim \sum_{i=1}^k H_i$.
    Since $d < 3k$, \Cref{lem:H_i^H_i} gives $\dim \sum_{i=1}^k H_i \wedge H_i \ge d+1$.
    Therefore we have
    \begin{align*}
     \E\sqbr{(N_3)_k} - \br{\gbinom{n}{3}}_k p^{-(n-3)3k}
     &\le \sum_{d=3}^{3k-1} \gbinom{n}{d} \gbinom{d}{3}^k p^{-(n-3)(d+1)} \\
    &\le 4^{k+1} \sum_{d=3}^{3k-1} p^{d(n-d)} p^{3(d-3)k} p^{-(n-3)(d+1)} \\
    &= 4^{k+1} \sum_{d=3}^{3k-1} p^{-n+3-9k + d(3k+3-d)} \\
    &\leq 4^{k+1} (3k-3) p^{-n+3-9k + 9(k+1)^2 / 4}
    \end{align*}
    since $d(3k+3 - d) \le 9(k+1)^2 / 4$ for all $d$.
    Since $p^{-n}$ tends to zero as $n\to\infty$, it follows that
    \[
        \E\sqbr{(N_3)_k} \sim \br{\gbinom{n}{3}}_k \>p^{-(n-3)3k} \to \lambda_p^k
       \]
    as $n \to \infty$, for $p$ and $k$ fixed, as desired.

    Finally, we prove (3) using the Bonferroni inequality
    \[
        \Pr\br{N_3 > 0} \> \le \> \sum_{k=1}^K (-1)^{k-1} \E\sqbr{\binom{N_3}{k}} ,
    \]
    which holds for every odd integer $K \ge 1$ (see~\cite{comtet}*{Section~4.7}).
    Moreover, we use the inequalities
    \begin{align*}
       \frac{1}{k!} \br{\gbinom{n}{3}}_k \>p^{-(n-3)3k}
        &\le \E\sqbr{\binom{N_3}{k}}
        &\le \frac{1}{k!} \br{\gbinom{n}{3}}_k p^{-(n-3)3k}
    + \frac{1}{k!} \sum_{d=3}^{3k-1} \gbinom{n}{d} \gbinom{d}{3}^k p^{-(n-3)(d+1)}
    \end{align*}
    that we obtained above.
    For $K = 3$, we have that
    \begin{align*}
        \Pr\br{N_3 + N_{\geqslant4} > 0}
        &\le \Pr\br{N_3 > 0} + \Pr\br{ N_{\geqslant4} > 0}\\
        &\le \sum_{k=1}^3 \frac{(-1)^{k-1}}{k!} \br{\gbinom{n}{3}}_k p^{-(n-3)3k} \\
        &\qquad + \frac{1}{6} \sum_{d=3}^{8} \gbinom{n}{d} \gbinom{d}{3}^3 p^{-(n-3)(d+1)} + 16(n-4) p^{4-n} .
    \end{align*}
    We can check using a computer that this is less than $1$, provided that $p \ge 13$ and $n \ge 12$.
    Alternatively, setting $K = 9$, bounding as above gives
    \begin{align*}
        \Pr\br{N_3 + N_{\geqslant4} > 0}
        &\le \sum_{k=1}^9 (-1)^{k-1} \E\sqbr{\binom{N_3}{k}} + 16 (n-4) p^{4-n} \\
        &\le \sum_{k=1}^9 \frac{(-1)^{k-1}}{k!} \br{\gbinom{n}{3}}_k p^{-(n-3)3k} \\
        &\qquad + \sum_{k=3,5,7,9} \frac{(3k-3)4^{k+1}}{k!} p^{-n+3-9k+9(k+1)^2/4}
         + 16 (n-4) p^{4-n}.
    \end{align*}
    We can check that this is less than $1$ for $2 \le p \le 11$ and $n \ge 160$.
    The tedious details of these calculations are recorded in \Cref{appendix-1}.
    This completes the proof.
\end{proof}

\begin{remark}[Lov\'asz local lemma]
    In the proof of \Cref{prop:alt-1}, it is tempting to try applying the Lov\'asz local lemma (see \cite{AS00}*{Sec.~5}) to the collection of events of the form ``$H$ is totally isotropic'', where $\dim H = 3$, since every event in that family is independent from almost all other others.
    However, the local lemma requires joint independence, not just pairwise independence, and it is easy to see that this fails in our situation.
\end{remark}

The following variant of \Cref{prop:alt-1} is much easier and completes the proof of \Cref{thExDMax} for $p \ne 2$. For $p = 2$ see \Cref{prop:alt-3}.

\begin{proposition} \label{prop:alt-2}
    Let $p \ne 2$ be a prime, $n \ge 2$, $V = \F_p^n$ and $W = \F_p^{n-2}$.
    Let $F \colon V \to W$ be a fixed surjective linear map,
    and let $B \in \Alt(V, W)$ be chosen uniformly at random.
    Call a proper subspace $H < V$ \emph{bad (for $B$)} if
    \[
           \dim W / (B(H, H) + F(H)) \> \ge \> \dim V / H.
    \]
    Then the probability that there is no bad subspace is positive for all $p$ and $n$.
    Moreover, it tends to $1$ as either $p, n \to \infty$.
\end{proposition}
\begin{proof}
    Note that a subspace $H < V$ is bad for $B$ if and only if
    \[
        \dim \br{B(H, H) + F(H)} \le \dim H - 2.
    \]
    On the other hand
    \begin{align*}
        \dim \br{B(H, H) + F(H)}
        \ge \dim F(H)
        &= \dim H - \dim (H \cap \ker F) \\
        &\ge \dim H - \dim (\ker F)
        = \dim H - 2.
    \end{align*}
    It follows that a subspace $H < V$ is bad for $B$ if and only if
    \[
        B(H,H) \leqslant F(H) \hspace{0.5cm} \mbox{and} \hspace{0.5cm} \ker F \leqslant H .
    \]
    Clearly we may assume $n \ge 3$.
    If $n = 3$, then the only possible bad subspace is the plane $\ker F$, so there is a bad subspace if and only if $B(\ker F, \ker F) =0$.
    We may therefore assume $n \ge 4$.
    For $2 \le h \le n-1$ let $N_h$ be the number of $h$-dimensional bad subspaces.
    For fixed $H < V$ with $\ker F \leqslant H$ and $\dim (H) = h$, we have
    \begin{align*}
        \Pr \br{B(H,H) \leqslant F(H)}
        = \frac{|\Alt(H, F(H))|}{|\Alt(H, W)|}
        = \frac{|\Alt(\F_p^h, \F_p^{h-2})|}{|\Alt(\F_p^h, \F_p^{n-2})|}
        = p^{-\binom{h}{2} (n - h)}.
    \end{align*}
    By summing over all choices of $H \geqslant \ker F$, it follows from \eqref{eq:gbinom-bounds} that
    \begin{align*}
        \E\sqbr{N_h} \le \gbinom{n-2}{h-2} p^{-\binom{h}{2}(n-h)}
        < 4 p^{(h-2)(n-h) - \binom{h}{2}(n-h)}
        = 4 p^{(h^2 - 3h + 4) (h-n)/2}.
    \end{align*}
    Let $f(h) = (h^2 - 3h + 4)(h-n)$. Then $f'(h) = (2h-3)(h-n) + h^2 - 3h + 4$, $f'(0) = 3n + 4 > 0$, $f'(2) = 4-n \le 0$.
    Since $f$ is a cubic function it follows that $f$ has a local maximum between $0$ and $2$ and no other local maximum. In particular the maximum value of $f$ in the range $3 \le h \le n-1$ is attained at one of the endpoints:
    \begin{align*}
        \max_{3 \le h \le n-1} (h^2-3h+4)(h-n)
        &= \max\left\{-4(n-3), -((n-1)^2 - 3(n-1)  + 4)\right\} \\
        &= -4(n-3).
    \end{align*}
    Let $N = \sum_{h=2}^{n-1} N_h$. Then, by Markov's inequality,
    \begin{align*}
        \Pr\br{N>0} \le \E\sqbr{N} = \sum_{h=2}^{n-1} \E\sqbr{N_h}
        &< 4 \sum_{h=2}^{n-1} p^{(h^2 - 3h+4)(h-n) / 2}
        &\le 4 p^{-(n-2)} + 4(n-3) p^{-2(n-3)}.
    \end{align*}
    This tends to zero as either $p, n \to \infty$,
    and we can easily check that it is smaller than $1$ for all $p \ge 3$ and $n \ge 4$.
\end{proof}

\begin{proposition} \label{prop:alt-3}
    Let $n \ge 2$, $V = \F_2^n$ and $W = \F_2^{n-2}$.
    Let $F \colon V \to W$ be a uniformly random quadratic map.
    Call a proper subspace $H < V$ \emph{bad (for $B$)} if
    \begin{equation} \label{eq:F-bad}
        \dim W / F(H) \> \ge \> \dim V / H.
    \end{equation}
    Then the probability that there is no bad subspace is positive for all $n$.
    Moreover, it tends to $1$ as either $n \to \infty$.
\end{proposition}
\begin{proof}
    A subspace $H < V$ is bad for $B$ if and only if $\dim F(H) \le \dim H - 2$.
    Clearly we may assume $n \ge 3$.
    For $2 \le h \le n-1$ let $N_h$ be the number of $h$-dimensional bad subspaces and let $N = \sum_{h=2}^{n-1} N_h$.
    For fixed $H < V$ and $K < W$ with $\dim H = h \ge 2$ and $\dim K = h - 2$, we have
    \[
        \Pr\br{F(H) \leqslant K} = \frac{2^{\frac{h(h+1)}{2} (h-2)}}{2^{\frac{h(h+1)}2 (n-2)}}
        = 2^{-h(h+1)(n-h)/2}.
    \]
    By summing over all choices of $H$ and $K$ it follows from \eqref{eq:gbinom-bounds} that
    \begin{align*}
        \E\sqbr{N_h} \le \gbinom{n}{h} \gbinom{n-2}{h-2} 2^{-h(h+1)(n-h)/2}
        < 16 \cdot 2^{(h^2 - 3h + 4) (h-n)/2}.
    \end{align*}
    By the same analysis of the previous proof we get
    \begin{align*}
        \Pr\br{N>0} < 16 \cdot 2^{-(n-2)} + 16(n-3) 2^{-2(n-3)}.
    \end{align*}
    As before this tends to zero as $n \to \infty$,
    and it is smaller than $1$ for $n \ge 7$.
    The cases $n = 3, 4, 5, 6$ can be checked on a computer with random sampling: see \Cref{appendix-2}.
\end{proof}

\section{Open problems}

We conclude by recalling some old open questions with similar flavour.

\begin{question}[Mann \cite{Man99}]
Is the derived length of $d$-maximal $2$-groups bounded?
\end{question}

It is not even known whether the nilpotency class of $d$-maximal $2$-groups can be unbounded.
The next question is equivalent to the following one.
Does there exist $c >0$ such that every finite solvable group $G$ contains an ab-maximal subgroup of size at least $|G|^c$?
See also \cite{Sab22}*{Conjecture 6}.

\begin{question}[Pyber \cite{Pyb97}]
Does there exist $c >0$ such that every finite $p$-group $P$ contains an abelian section of size at least $|P|^c$?
\end{question}

Let $m = m(n)$ be the largest integer such that every group of order $p^n$ has an abelian subgroup of order $p^m$.
From results of Burnside~\cite{burnside} and Ol'shanskii~\cite{Ols78} we know $\sqrt{2n} + o(\sqrt{n}) \leq m(n) \leq 2\sqrt{2n}$.

\begin{question}[Ol'shanskii~\cite{Ols78}]
    Is there a constant $c$ such that $m \sim c \sqrt{n}$?
\end{question}

\bibliography{refs}

\appendix

\section{Calculations}

For calculations we used Sage (The Sage Mathematics Software System (Version 6.9), The Sage Developers, 2022, \href{https://www.sagemath.org}{https://www.sagemath.org}).

\subsection{Proof of \Cref{prop:alt-1}}
\label{appendix-1}

We claim that
\begin{equation}
    \label{eq:exact-bound-1}
    \sum_{k=1}^3 \frac{(-1)^{k-1}}{k!} \br{\gbinom{n}{3}}_k p^{-(n-3)3k} + \frac{1}{6} \sum_{d=3}^{8} \gbinom{n}{d} \gbinom{d}{3}^3 p^{-(n-3)(d+1)} + 16(n-4) p^{4-n} < 1
\end{equation}
for all $p \ge 13$ and $n \ge 12$.
Define
\[
    \lambda_p(n) = \gbinom{n}{3} p^{-3(n-3)}
    = \lambda_p (1 - p^{-n}) (1 - p^{1-n}) (1 - p^{2-n}).
\]
Using \eqref{eq:gbinom-bounds} several times, we can bound the left-hand side of \eqref{eq:exact-bound-1} by
\[
    \lambda_p(n) - \frac12 \lambda_p(n)^2 + \frac16 \lambda_p(n)^3 + 2 p^{-3(n-3)}
    + \frac{(1-p^{-1}-p^{-2})^{-4}}{6} \sum_{d=3}^{8} p^{-d^2 + 12d - 24 - n}
    + 16(n-4) p^{4-n}.
\]
The unique maximum value of $-d^2+12d-24$ is $12$, attained at $d=6$.
Also, $\lambda_p(n) \le \lambda_p$, and the cubic polynomial $F_3(x) = x - x^2 / 2 + x^3 / 3$ is monotonic increasing for all $x$.
Therefore we obtain the bound
\begin{equation}
    \label{eq:crude-bound-1}
    F_3(\lambda_p)
    + 2 p^{-3(n-3)}
    + \frac{(1-p^{-1}-p^{-2})^{-4}}{6} \br{p^{12 - n} + 5 p^{11-n}} + 16(n-4) p^{4-n},
\end{equation}
which is again decreasing in $p$ for $n \ge 12$. Therefore it suffices to check that \eqref{eq:crude-bound-1} is less than $1$ for $(p, n) = (13, 15)$ and $(p,n) = (17, 12)$ and to check \eqref{eq:exact-bound-1} directly for $p = 13$ and $12 \le n \le 14$.

Next, we claim that
\begin{align}
    \sum_{k=1}^9 \frac{(-1)^{k-1}}{k!} \br{\gbinom{n}{3}}_k p^{-(n-3)3k} + \sum_{k=3,5,7,9} \frac{(3k-3)4^{k+1}}{k!} p^{-n+3-9k+9(k+1)^2/4}
    + 16 (n-4) p^{4-n} < 1\label{eq:exact-bound-2}
\end{align}
for $2 \le p \le 11$ and $n \ge 160$.
Observe that
\[
    \br{\gbinom{n}{3}}_k p^{-(n-3)3k} \le \lambda_p(n)^k
\]
and
\[
    \br{\gbinom{n}{3}}_k p^{-(n-3)3k} \ge \lambda_p(n)^k \br{ 1- \binom{k}{2} \gbinom{n}{3}^{-1}}
    \ge \lambda_p(n)^k - k^2 \lambda_p^k p^{9-3n},
\]
so we obtain the bound
\[
    F_9(\lambda_p(n)) + \sum_{k=3,5,7,9} \br{\frac{k^2 \lambda_p^k}{k!} p^{9-3n} + \frac{(3k-3)4^{k+1}}{k!} p^{147-n}} + 16(n-4) p^{4-n},
\]
where $F_9$ is the polynomial $F_9(x) = \sum_{k=1}^9 (-1)^{k-1} x^{k} / k!$. Since $F_9(x)$ is monotonic increasing for $0 < x < 4$ and $\lambda_p(n)$ is monotonic increasing in $n$ and decreasing in $p$, we have $F_9(\lambda_p(n)) \le F_9(\lambda_2)$,
and the other terms are obviously largest for $p=2$.
Bounding the middle terms crudely, we obtain the bound
\begin{equation}
    \label{eq:crude-bound-2}
    F_9(\lambda_2) + 4^{11} (2^{9-3n} + 2^{147-n}) + 16(n-4) 2^{4-n}.
\end{equation}
To finish it suffices to check that \eqref{eq:crude-bound-2} is less than $1$ for $n = 180$ and to check \eqref{eq:exact-bound-2} directly for $2 \le p \le 11$ and $160 \le n \le 180$.

To check all these statements we used the code below.

\begin{lstlisting}[language=Python]
def lamb(p):
    return 1/((1-1/p) * (1-1/p^2) * (1-1/p^3))

def F(k, x):
    return sum((-1)^(i-1) * x^i / factorial(i) for i in [1..k])

def boundA1(p, n):
    return sum((-1)^(k-1) * binomial(gaussian_binomial(n, 3, p), k) * p^(-(n-3)*3*k) for k in [1..3]) + 1/6 * sum(gaussian_binomial(n, d, p) * gaussian_binomial(d, 3, p)^3 * p^(-(n-3)*(d+1)) for d in [3..8]) + 16 * (n-4) * p^(4-n)

def boundA2(p, n):
    return F(3, lamb(p)) + 2 * p^(-3*n + 9) + (1-1/p-1/p^2)^(-4) * (1/6 + 1/p) * p^(12-n) + 16 * (n-4) * p^(4-n)

def boundA3(p, n):
    return sum((-1)^(k-1) * binomial(gaussian_binomial(n, 3, p), k) * p^(-(n-3)*3*k) for k in [1..9]) + sum((3*k-3)*4^(k+1) / factorial(k) * p^(-n+3-9*k+9*(k+1)^2/4) for k in [3,5,7,9]) + 16 * (n-4) * p^(4-n)

def boundA4(n):
    return F(9, lamb(2)) + 4^11 * (2^(9-3*n) + 2^(147-n)) + 16 * (n-4) * 2^(4-n)

print(max([boundA1(13, n) for n in [12, 13, 14]] + [boundA2(13, 15), boundA2(17, 12)] + [boundA3(p, n) for p in primes(2, 12) for n in [160..180]] + [boundA4(180)]) < 1)
\end{lstlisting}

\subsection{Proof of \Cref{prop:alt-3}}
\label{appendix-2}

For $n = 3, 4, 5, 6$ we claim that there is at leat one quadratic map $F : \F_2^n \to \F_2^{n-2}$ such that $\dim F(H) \ge \dim H - 1$ for every proper subspace $H < V$. The following Sage program checks this by random sampling.

\begin{lstlisting}[language=Python]
F = GF(2)

def random_quad_form(n):
    B = Matrix.random(F, n, n)
    for i in range(n):
        for j in range(i):
            B[i, j] = B[j, i]
    return B

def random_quadratic_map(n, m):
    return [random_quad_form(n) for _ in range(m)]

def quadratic_map_image(B, subspace):
    basis = subspace.basis()
    d, m = len(basis), len(B)
    W = F**m
    return W.subspace([W([x * Bi * y for Bi in B]) for x in basis for y in basis])

for n in [2..6]:
    V = F ** n
    while True:
        B = random_quadratic_map(n, n-2)
        if all(quadratic_map_image(B, H).dimension() >= h - 1 for h in [1..n-1] for H in V.subspaces(h)):
            break
    print(n, 'pass')
\end{lstlisting}

\end{document}